\documentclass[12pt]{amsart}

\usepackage{verbatim, amssymb, enumitem, mathtools}

\usepackage[breaklinks=true,colorlinks=true,linkcolor=blue,citecolor=red,urlcolor=blue,psdextra,pdfencoding=auto]{hyperref}

\newcommand \ve{\varepsilon}

\newcommand \br{\mathbb{R}}

\newcommand \rk{\operatorname{rk}}
\newcommand \Ker{\operatorname{Ker}}
\newcommand \Der{\operatorname{Der}}

\newcommand \Span{\operatorname{Span}}

\newcommand \cB{\mathcal{B}}

\newcommand\ag{\mathfrak a}
\newcommand\kg{\mathfrak k}
\newcommand\g{\mathfrak g}
\newcommand\cs{\mathfrak c}
\newcommand\h{\mathfrak h}
\newcommand\z{\mathfrak z}
\newcommand\m{\mathfrak m}
\newcommand \so{\mathfrak{so}}

\newcommand \gl{\mathfrak{gl}}
\newcommand \ug{\mathfrak{u}}

\newcommand \s{\mathfrak{s}}
\newcommand \vg{\mathfrak{v}}
\newcommand \n{\mathfrak{n}}

\newcommand \p{\mathfrak{p}}
\renewcommand\t{\mathfrak t}

\newcommand \ad{\operatorname{ad}}

\newcommand \<{\langle}
\renewcommand \>{\rangle}
\newcommand \ip{\<\cdot,\cdot\>}

\newtheorem{theorem}{Theorem}
\newtheorem*{theorem*}{Theorem}

\newtheorem*{corollary*}{Corollary}
\newtheorem*{conj*}{Conjecture}

\newtheorem{proposition}{Proposition}
\newtheorem*{prop*}{Proposition}

\theoremstyle{definition}

\newtheorem*{definition*}{Definition}

\theoremstyle{remark}

\newtheorem*{notation*}{Notation}
\newtheorem*{algorithm*}{Algorithm}
\newtheorem*{example*}{Example}

\makeatletter
\@namedef{subjclassname@2020}{%
  \textup{2020} Mathematics Subject Classification}
\makeatother

\begin{document}

\title{The Structure of Geodesic Orbit Lorentz Nilmanifolds}

\author{Yuri Nikolayevsky}
\address{Department of Mathematical and Physical Sciences, La Trobe University, VIC 3086, Australia} 
\email{Y.Nikolayevsky@latrobe.edu.au}
\thanks{The first author was partially supported by ARC Discovery Grant DP210100951. }

\author{Joseph A. Wolf}
\address{Department of Mathematics, University of California, Berkeley, CA 94720-3840, USA}
\email{jawolf@math.berkeley.edu}
\thanks{The second author was partially supported by a Simons Foundation grant}
\thanks{Both authors thank the mathematical research institute MATRIX in 
Australia where part of this research was performed}



\begin{abstract}
The geodesic orbit property is useful and interesting in Riemannian geometry.
It implies homogeneity and has important classes of Riemannian manifolds as
special cases.  Those classes include weakly symmetric Riemannian manifolds
and naturally reductive Riemannian manifolds.  The corresponding results for
indefinite metric manifolds are much more delicate than in Riemannian 
signature, but in the last few years 
important corresponding structural results were proved for geodesic 
orbit Lorentz manifolds.  Here we carry out a major step in the 
structural analysis of geodesic orbit Lorentz nilmanifolds.  Those are the 
geodesic orbit
Lorentz manifolds $M = G/H$ such that a nilpotent analytic subgroup of $G$
is transitive on $M$.  Suppose that there is a reductive decomposition 
$\g = \h \oplus \n$ (vector space direct sum) with $\n$ nilpotent. 
When the metric is nondegenerate on $[\n,\n]$ we 
show that $\n$ is abelian or 2-step nilpotent (this is the same result as for 
geodesic orbit Riemannian nilmanifolds), and when the metric
is degenerate on $[\n,\n]$ we show that $\n$ is a Lorentz 
double extension corresponding to a geodesic orbit Riemannian nilmanifold.  
In the latter case we construct examples to 
show that the number of nilpotency steps is unbounded.

\end{abstract}

\maketitle

\section{Introduction}
\label{intro}

A pseudo-Riemannian manifold $(M,ds^2)$ is called a \emph{geodesic orbit 
manifold} (or a manifold with homogeneous geodesics, or simply a $GO$ 
manifold), if every geodesic of $M$ is an orbit of a $1$-parameter subgroup 
of the full isometry group $I(M) = I(M,ds^2)$.  One loses no generality 
if one replaces $I(M)$ by 
its identity component $I^0(M)$.  If $G$ is a transitive Lie subgroup of
$I^0(M)$, so $(M,ds^2) = (G/H,ds^2)$ where $H$ is an isotropy subgroup of
$G$, and if every geodesic of $M$ is an orbit of a 
$1$-parameter subgroup of $G$, then we say that $(M,ds^2)$ is a 
\emph{$G$-geodesic orbit manifold}, or a $G$-$GO$ manifold.  Clearly every 
$G$-$GO$ manifold is a $GO$ manifold, but not vice versa.  The class of 
geodesic orbit manifolds includes 
(but is not limited to) symmetric spaces, weakly symmetric spaces, normal 
and generalized normal homogeneous spaces, and naturally reductive spaces. For 
the current state of knowledge in the theory of Riemannian geodesic orbit 
manifolds we refer the reader to \cite{BN} and its bibliography.

In this paper, we study the $GO$ condition for pseudo-Riemannian nilmanifolds
$(N,ds^2)$, relative to subgroups $G \subset I(N)$ of the form 
$G = N \rtimes H$, where $H$ is an isotropy subgroup.  Most of our results
apply to the case where $(N,ds^2)$ is a Lorentz manifold.

Our results for $G$-$GO$ manifolds $(M,ds^2) = (G/H,ds^2)$ 
require the coset space $G/H$ to be reductive. In other words they make use of
an ${\rm Ad}_G(H)$--invariant decomposition $\g = \m \oplus \h$.  Very few 
structural results are known for indefinite metric $GO$ manifolds that are 
not reductive, and we always assume that $G/H$ is reductive.

Recall that a 
pseudo-Riemannian \emph{nilmanifold} is a pseudo-Riemannian manifold 
admitting a transitive nilpotent Lie group of isometries.  
In the Riemannian case, the full isometry group of a nilmanifold $(N,ds^2)$, 
where $N$ is a transitive nilpotent group of isometries, is the semidirect
product $I(N) = N \rtimes H$, where $H$ is the group of all isometric 
automorphisms of $(N,ds^2)$ \cite[Theorem~4.2]{W1963}. In other words, $N$
is the nilradical of $I(N)$.  In the pseudo-Riemannian 
cases, $I(N)$ might still contain $N \rtimes H$ and yet be strictly larger. In 
indefinite signatures of metric a nilmanifold is not necessarily reductive 
as a coset space of $I(N)$, and even when it is, $N$ does not have to be 
a normal subgroup of $I(N)$.  Here the $GO$ condition does not rescue us, for
there exist $4$-dimensional, Lorentz $GO$ nilmanifolds that are reductive 
relative to $I(N)$, but for which $N$ is not an ideal in $I(N)$ 
\cite[Section~3]{dBO}. Moreover, already in dimension $4$ (the lowest 
dimension for homogeneous pseudo-Riemannian spaces $G/H$ with 
$H$ connected that are not reductive), every non-reductive space is a 
$GO$ manifold when we make 
a correct choice of parameters \cite[Theorem~4.1]{CFZ}. A complete 
classification of pseudo-Riemannian $GO$ manifolds of dimension $4$ is given 
in \cite{CZ}.

In Section \ref{s:prel} we recall some basic facts on reductive geodesic 
orbit spaces.  In particular the Geodesic Lemma (recalled as 
Proposition \ref{GeoL} below) gives an
algebraic condition \newline $\<[T+A,T']_\m,T\> = k \<T,T'\>$ for a reductive
pseudo-Riemannian homogeneous space $M=G/H$, with $\g = \h \oplus \m$,
to be $GO$. We also recall the notion of geodesic graph and use it in 
Proposition \ref{p:natred} for a characterization of the naturally reductive 
condition.

In Section \ref{theorem1} we sharpen \cite[Theorem~7]{CWZ} to obtain a
basic structure result on reductive $GO$ Lorentz nilmanifolds $(G/H,ds^2)$.
Write $\g = \h \oplus \n$ with $\n$ nilpotent and let $\ip$ denote the
inner product on $\n$ defined by $ds^2$.  If $\ip|_{[\n,\n]}$
is nondegenerate then \cite[Theorem~7]{CWZ} says $N$ is either abelian, or 
$2$-step nilpotent, or $4$-step nilpotent. While there are many
examples of abelian and of $2$-step nilpotent, there were no examples of
$4$-step nilpotent.  Our Theorem \ref{th:lorentzdernondeg}
eliminates the $4$-step possibility.  That is the main result of this
paper. Theorem~\ref{th:lorentzdernondeg} in Section~\ref{theorem1} is 
obtained as a corollary of 
Theorem~\ref{th:dernondeg}; the latter is valid in higher signatures provided 
the derived algebra of $\n$ is either abelian or Lorentz.

In Section \ref{theorem2} we recall the notion of double extension and
use it to obtain a complement, Theorem \ref{th:derdeg}, to 
Theorem~\ref{th:lorentzdernondeg}.  While Theorem~\ref{th:lorentzdernondeg} 
requires that $\ip|_{[\n,\n]}$ be nondegenerate, Theorem \ref{th:derdeg}
requires that it be degenerate, and then it shows that $(G/H,ds^2)$
is a Lorentz double extension of a Riemannian $GO$ nilmanifold.

In Section \ref{examples} we construct a large family of naturally
reductive
$GO$ Lorentz nilmanifolds $(G/H,ds^2)$ that are double extensions of 
Riemannian $GO$ nilmanifolds. There the 
transitive nilpotent groups are $r$-step nilpotent for unbounded $r$.
Theorem \ref{th:examples} extracts a few of those double extension
manifolds and shows that for every $d > 0$ there is a naturally reductive
Lorentz nilmanifold $(N,ds^2)$ of nilpotent step $\geq d$ and dimension $d+4$,
and a corresponding Lorentz nilmanifold $(N,ds^2)$ of nilpotent step $\geq d$
and dimension $d+10$ that is not naturally reductive.  

\section{Preliminaries}
\label{s:prel}

Let $M=G/H$ be a pseudo-Riemannian homogeneous space. As usual $\g$ and $\h$ 
denote the Lie algebras of $G$ and $H$.  In the Riemannian case there is 
an $\ad_\g(\h)$-module $\m$ such that $\g = \h \oplus \m$ as a linear space. 
This is not necessarily true in an arbitrary signature; if it is, the 
pseudo-Riemannian homogeneous space $M=G/H$ is called \emph{$G$-reductive}. 
Note that reductivity depends on the choice of the isometry group $G$. 
Any corresponding decomposition $\g = \h \oplus \m$ is called a 
\emph{reductive decomposition}.

The $GO$ condition for reductive spaces is given in the {\bf Geodesic Lemma}:
\begin{proposition} \label{GeoL}
Let $M=G/H$ be a reductive pseudo-Riemannian homogeneous space, with reductive 
decomposition $\g = \h \oplus \m$. Then $M$ is a $G$-geodesic orbit space 
if and only if, for any $T \in \m$, there exist $A = A(T) \in \h$ and 
$k=k(T) \in \br$ such that if $T' \in \m$ then
\begin{equation}\label{eq:golemma}
  \<[T+A,T']_\m,T\> = k \<T,T'\>.
\end{equation}
The subscript ${}_\m$ means the $\m$-component in $\g = \h + \m$.
\end{proposition}
Substituting $T'=T$ one sees that $k(T)=0$ unless $T$ is a null vector. 
In particular, $k$ is always zero in the Riemannian signature. Any map 
$A: \m \to \h$ for which \eqref{eq:golemma} holds (with some function $k$) 
is called a \emph{geodesic graph}. If a geodesic graph exists (that is, 
if the space is $GO$), it can be chosen $\ad_\g(\h)$-equivariant, i.e. such that
$[L, A(T)]=A([L,T])$, for all $L \in \h$ and all $T \in \m$.

A pseudo-Riemannian homogeneous space $M=G/H$ is ($G$-)\emph{naturally 
reductive} if there is a reductive decomposition $\g = \h \oplus \m$ such 
that 
\begin{equation}\label{defnatred}
\text{if } T, T' \in \m \text{ then } \<[T', T]_\m, T\>=0. 
\end{equation}
In our nilmanifold case, if $G = N \rtimes H$, the space $G/H$ might be 
naturally reductive using a choice of a complementary $\h$-module $\m$ 
different from $\n$. In the case $G = N$ the natural reductivity condition 
says that the inner product on $\n$ is invariant under the adjoint
representation, so the metric on $N$ is \emph{bi-invariant}, in other
words invariant under both left and right translations. 
Kostant's criterion for natural reductivity in the Riemannian signature 
\cite[Theorem~4]{Kos} is valid as well in pseudo-Riemannian case 
\cite[Theorem~2.2]{Ova1}.

The property of being naturally reductive depends on the choice of group 
$G$ in the presentation $M=G/H$: both enlarging and reducing $G$ may lead 
to gaining or losing the natural reductivity property, even in the 
Riemannian setting \cite[\S~2]{DZ}.  This contrasts with the $GO$ condition, 
which is trivially preserved under enlarging the isometry group. 

A $G$-naturally reductive space is always $G'$-geodesic orbit for any 
$G' \supset G$. This is seen by taking $A=0$ and $k=0$ in the Geodesic Lemma. 
The converse fails even in Riemannian case, where there are $GO$ spaces that 
are not $G$-naturally reductive for any choice of the transitive group 
$G$ (\cite[Proposition~3]{Kap}, \cite[Theorem~5.3(I)]{KV}). Also, see
Theorem \ref{th:examples} below.

Proposition \ref{p:natred} below is useful for deciding whether a $GO$ space 
is naturally reductive relative to the same group $G$. The proof is essentially
the same as in \cite[Corollaire~2, Lemme~10]{Sze} for the affine case 
(although the ``only if" direction there requires $H$ to be compact), and 
in \cite[Proposition~2.10]{KV} for the Riemannian case. For completeness we 
include it below.

\begin{proposition} \label{p:natred} 
Let $M=G/H$ be a reductive $G$-$GO$ space with reductive decomposition 
$\g = \h \oplus \m$. Then $M$ is $G$-naturally reductive if and only if a 
geodesic graph $A: \m \to \h$ in the Geodesic Lemma can be chosen linear 
and $\ad_\g(\h)$-equivariant. 
\end{proposition}

\begin{proof}
Let $M=G/H$ be a naturally reductive $G$-$GO$ space. Choose a reductive 
decomposition $\g=\h\oplus \m$ such that $\<[T', T]_\m, T\>=0$ for all 
$T, T' \in \m$. Let $\g=\h\oplus \p$ be any other reductive decomposition. 
Since both $\m$ and $\p$ are naturally identified with the tangent space of 
$(M,ds^2)$ at the base point, there is a uniquely defined 
$\ad_\g(\h)$-invariant 
isometry $\iota: \m \to \p$. Now for the decomposition $\g=\h\oplus \p$ the 
equation \eqref{eq:golemma} holds with $A(X)=\iota^{-1}X-X$ and $k(X)=0$, for 
all $X \in \p$. 

Conversely, given a reductive decomposition $\g=\h\oplus \p$ with an 
$\ad_\g(\h)$-equivariant linear map $A: \p \to \h$ such that \eqref{eq:golemma} 
holds for $X \in \p$ (forcing $k=0$ by continuity), we define 
$\m=\Span(X+A(X) \, | \, X \in \p)$, with the inner product such that the 
map $X \mapsto X + A(X)$ is an isometry. Then $\g=\h \oplus \m$, and $\m$ is 
an $\h$-module because $A$ is $\ad_\g(\h)$-equivariant. It is easy to check 
that $\<[T', T]_\m, T\>=0$ for all $T, T' \in \m$.
\end{proof}

In the Riemannian case, or more generally when $H$ is compact, the 
existence of a linear geodesic graph implies the existence of a linear, 
$\ad_\g(\h)$-equivariant geodesic graph, and hence is equivalent to natural 
reductivity \cite[Lemma~3]{CNN}.

\section{Eliminating \ensuremath{4}-Step When \ensuremath{[\texorpdfstring{\n}{\unichar{"1D52B}},\texorpdfstring{\n}{\unichar{"1D52B}}]} is Nondegenerate}
\label{theorem1}

Some years ago Gordon proved \cite[Theorem~2.2]{Gor} that a Riemannian $GO$ 
nilmanifold is at most $2$-step nilpotent.  More recently Chen, Wolf and
Zhang proved \cite[Theorem 7]{CWZ} that a connected Lorentz $G$-geodesic orbit
nilmanifold $M = G/H$, with $G = N \rtimes H$, $N$ nilpotent and
$\ip|_{[\n,\n]}$ nondegenerate\footnote{Nondegeneracy of $\ip|_{[\n,\n]}$
is stated in the paragraph before the statement of \cite[Theorem 7]{CWZ}
 and is recalled and used in the proof, but perhaps it could have been part 
of the statement itself.}, has
similar properties:  $\n$ is abelian, or $\n$ is $2$-step nilpotent, or 
$\n$ is $4$-step nilpotent.  Our theorems here eliminate the $4$-step 
possibility and go somewhat beyond the case of Lorentz signature.

Given a reductive homogeneous pseudo-Riemannian manifold $(G/H,ds^2)$,
where $G = N \rtimes H$ with $N$ nilpotent, identify 
$\n = {\rm Lie}(N)$ with the tangent space to $G/H$ at $1N$.

\begin{theorem}\label{th:dernondeg}
     Let $(M = G/H, ds^2)$ be a connected pseudo-Riemannian $G$-geodesic 
orbit nilmanifold where $G = N \rtimes H$ with $N$ nilpotent. Denote 
$\ip'=\ip|_{[\n,\n]}$. In each of the following cases, $N$ is either abelian 
or $2$-step nilpotent:
     \begin{enumerate}[label=\emph{(\alph*)},ref=\alph*]
	\item \label{it:n'd}
	$\ip'$ is of definite signature;
	\item \label{it:n'lcnon}
	$\ip'$ is of Lorentz signature and the centralizer of $[\n,\n]$ is non-degenerate;
	\item \label{it:n'lnl}
	both $\ip$ and $\ip'$ are of Lorentz signature.
     \end{enumerate}
\end{theorem}

\begin{proof}
The first part of our argument is similar to part of the proof of 
\cite[Theorem~7]{CWZ}. Let $\n'=[\n,\n]$, and denote $m = \dim \n'$. 
We can assume that $m \ge 2$. 

Suppose that $\ip'$ is nondegenerate. Let $\vg$ denote the orthogonal 
complement to $\n'$ in $\n$. Then we have the $\ad_\g(\h)$-invariant 
orthogonal direct sum decomposition $\n=\n' \oplus \vg$. Let $T=X+Y$ 
and $T' = X'+Y'$ where $X,X' \in \n'$, $Y,Y' \in \vg$, and $T$ is non-null 
in \eqref{eq:golemma}.  Then $k(T)=0$ and we have $A=A(X,Y) \in \h$ such that
  \begin{equation}\label{eq:gonondeg}
  \<[A,X'],X\> + \<[A,Y'],Y\> + \<[X,X']+[Y,X']+[X,Y']+[Y,Y'],X\>=0.
  \end{equation}
Taking $Y'=Y, X'=0$ we obtain, by continuity,
  \begin{equation}\label{eq:YXX}
    \<[Y,X],X\>=0, \quad\text{for all } Y \in \vg, \, X \in \n'.
  \end{equation}
As $\vg$ generates $\n$ it follows that   
  \begin{equation}\label{eq:TXX}
    \<[T,X],X\>=0, \quad\text{for all } T \in \n, \, X \in \n'.
  \end{equation}
  
Separating the $X'$- and the $Y'$-components in \eqref{eq:gonondeg} and 
using \eqref{eq:YXX} and \eqref{eq:TXX}, we find that for all $X \in \n'$ 
and $Y \in \vg$ with $X+Y$ non-null, there exists $A=A(X,Y) \in \h$ such 
that for all $X' \in \n', \, Y' \in \vg$, 
  \begin{gather}\label{eq:AYY}
    \<[A,Y],Y'\> = \<[Y,Y'],X\>,\\
    [A+Y,X] =0. \label{eq:AYX}
  \end{gather}

Assertion~\eqref{it:n'd} now follows from~\eqref{eq:TXX}: for all $T \in \n$, 
the (nilpotent) operator $\ad_\g(T)|_{\n'}$ on $\n'$ is skew-symmetric relative 
to the definite inner product $\ip'$, and hence is zero, which implies 
$[\n,\n']=0$.  (This follows the argument of \cite[Theorem~2.2]{Gor}.) 

For \eqref{it:n'lcnon} and \eqref{it:n'lnl} we suppose that $\ip'$ 
is Lorentz. Denote $\s := \so(m -1, 1) \subset \gl(\n')$, the algebra of 
skew-symmetric endomorphisms relative to $\ip'$. Thus 
$\kg:= \ad_\g(\n)|_{\n'}$ is a subalgebra of $\s$ consisting of nilpotent 
endomorphisms.  As such,
using Engel's Theorem, it is triangular.  Thus it is conjugate by an inner 
automorphism \cite[Theorem~2.1]{Mos} to a subalgebra of the nilpotent part 
$\ug$ of an Iwasawa decomposition $\s=\t \oplus \ag \oplus \ug$.  Now we may 
(and do) assume $\kg \subset \ug$. 

Choose a basis $\{f_1, \dots, f_m\}$ for $\n'$ such that $\<f_i,f_j\>= \ve_i \delta_{ij}$, where $\ve_1=-1$ and $\ve_i = +1$ for $i>1$, and such that the maximal compact subalgebra $\t=\so(m-1)$ acts on $\Span(f_2,\dots,f_m)$, and the $1$-dimensional abelian subalgebra is given by $\ag = \left(\begin{smallmatrix} 0 & 1 & 0 \\ 1 & 0 & 0 \\ 0 & 0 & 0_{m-2} \end{smallmatrix}\right)\br$.  Here $0_{m-2}$ denotes the $(m-2)\times(m-2)$
zero matrix. Then $\ug$ is the space of matrices of the form $\left(\begin{smallmatrix} 0 & 0 & u^t \\ 0 & 0  & u^t\\ u & -u & 0_{m-2}
\end{smallmatrix}\right)$ where $u \in \br^{m-2}$. We introduce a new basis
for $\n'$ given by $e_1=(f_1+f_2)/\sqrt2, \, e_2=(f_1-f_2)/\sqrt2$ and
$e_i=f_i$ for $i > 2$. Relative to this basis, we have
  \begin{equation*}
    \ip|_{\n'} = \left ( \begin{smallmatrix} 0 & -1 & 0 \\ -1 & 0 & 0 \\
	0 & 0 & I_{m-2} \end{smallmatrix} \right ) \text{ and }
	\ug = \left\{ \left ( \begin{smallmatrix} 0 & 0 & u^t \\ 0 & 0 & 0 \\
	0 & u & 0_{m-2} \end{smallmatrix}\right ), \; u \in \br^{m-2} \right\}.
  \end{equation*}
As $\kg \subset \ug$, we obtain a linear map
$\Phi: \vg \to \Span(e_3, \dots, e_m)$ such that, for all $Y \in \vg$,
 \begin{equation}\label{eq:advg}
    \ad_\g(Y) e_1=0,\; \ad_\g(Y) e_2 = \Phi Y, \text{ and }
	\ad_\g(Y) e_i = \<\Phi Y,e_i\>e_1 \text{ for } i > 2.
 \end{equation}
As $\ug$ (and hence $\kg$) is abelian, $[[\vg,\vg],\n']=0$. Since $\vg$ generates $\n$, we obtain $[[\n,\n],\n']=0$, which implies that $\n'$ is
abelian. 

Now assertion~\eqref{it:n'lcnon} implies assertion~\eqref{it:n'lnl}. Indeed, 
the centralizer of $\n'$ is the direct sum of $\n'$ and a subspace $\cs$ of 
$\vg$, and if $\ip$ is Lorentz the inner product on $\vg$ is definite, which 
implies that it is also definite on $\cs$, and hence the centralizer of $\n'$ 
is nondegenerate.

To complete the proof of the theorem, it remains to establish 
assertion~\eqref{it:n'lcnon}. 

Introduce $2$-forms $\omega_i \in \Lambda^2(\vg)$ by
  \begin{equation}\label{eq:omega}
    [V_1,V_2]= \sum\nolimits_{i=1}^{m} \omega_i(V_1,V_2)e_i
	\text{  for  } V_1, V_2 \in \vg.
  \end{equation}
As $e_1 \in \n'$ and $\ip'$ is nondegenerate, we cannot have $e_1 \perp \n'$. But since $\n'$ is abelian we get $\n'= [\vg,\vg]+ [\vg,\n']$, and from \eqref{eq:advg} we obtain $e_1 \perp [\vg,\n']$.  Thus $e_1 \not\perp [\vg,\vg]$, which by \eqref{eq:omega} implies $\omega_2 \ne 0$.

Using \eqref{eq:advg} and \eqref{eq:omega}, the Jacobi identity gives
  \begin{equation}\label{eq:Jac}
    \sigma\Big(\omega_2(V_1,V_2)\Phi V_3 + \sum\nolimits_{i=3}^{m}
	\omega_i(V_1,V_2)\<\Phi V_3,e_i\>e_1\Big)=0,
  \end{equation}
where $\sigma$ denotes the cyclic permutation of $V_1,V_2,V_3 \in \vg$.
If $\rk \Phi \ge 3$, then for almost all triples $V_1,V_2,V_3 \in \vg$,
the vectors $\Phi V_1, \Phi V_2$ and $\Phi V_3 \in \Span(e_3, \dots, e_m)$ are linearly independent, so $\omega_2=0$ by
\eqref{eq:Jac}.  This is a contradiction, so $\rk \Phi \le 2$.

If $\Phi =0$, the algebra $\n$ is $2$-step nilpotent by \eqref{eq:advg}. Suppose $\Phi \ne 0$. 

The centralizer of $\n'$ in $\n$ is an $\ad_\g(\h)$-invariant ideal in $\n$, so its intersection with $\vg$, which is the subspace $\cs =\Ker \Phi \subset \vg$, is also $\ad_\g(\h)$-invariant. Then the subspace $\cs^\perp\subset\vg$ is $\ad_\g(\h)$-invariant as well. Note that $\dim \cs^\perp = \rk \Phi \in \{1,2\}$ by the above argument, and that $\ip|_{\cs}$ is nondegenerate by our assumption. Then $\ip|_{\cs^\perp}$ is also nondegenerate. As $\Phi\cs =0$, \eqref{eq:Jac} implies $\omega_2(\cs,\cs)=0$, where we take $V_1, V_2 \in \cs$ and $V_3 \in \cs^\perp$. Moreover, taking in \eqref{eq:AYY} $Y \in \cs, \, Y' \in \cs^\perp$ we obtain $[\cs, \cs^\perp]=0$, and in particular, $\omega_2(\cs,\cs^\perp)=0$. 

Since $\omega_2 \ne 0$, we must have $\omega_2(\cs^\perp,\cs^\perp) \ne 0$, 
and so $\dim \cs^\perp = 2$. Let $Y_1, Y_2$ be a basis for $\cs^\perp$ such 
that $\<Y_i, Y_j\> = \ve_i \delta_{ij}$, where $\ve_i = \pm 1$ for $i=1,2$. 
As $\cs^\perp$ is $\ad_\g(\h)$-invariant, $\ad_\g(\h)|_{\n}$ is skew-symmetric, 
we obtain $[A, Y_1]= \ve_1 \mu(A) Y_2$ and $[A, Y_2] = -\ve_2 \mu(A) Y_1$,
for any $A \in \h$, where $\mu$ is a linear functional on  
$\h$. In particular, $[A,[Y_1,Y_2]]=0$ for all $A \in \h$. Taking $X=[Y_1,Y_2]$ 
in \eqref{eq:AYX} we see that $[Y,[Y_1,Y_2]]=0$ for all $Y \in \vg$ such that 
$Y+[Y_1,Y_2]$ is non-null. Thus $[Y,[Y_1,Y_2]]=0$ for all $Y \in \vg$. But then 
\eqref{eq:advg} and \eqref{eq:omega} imply $\omega_2(Y_1,Y_2)\Phi Y=0$, 
so $\omega_2(Y_1,Y_2)=0$. So $\omega_2(\cs^\perp,\cs^\perp) = 0$, which is a 
contradiction.
\end{proof}

The Lorentz manifold case of Theorem \ref{th:dernondeg} is of special interest,
so we state it separately.

\begin{theorem}\label{th:lorentzdernondeg} 
Let $(M = G/H, ds^2)$ be a connected Lorentz $G$-geodesic orbit nilmanifold 
where $G = N \rtimes H$ with $N$ nilpotent.  Let $\ip$ denote the inner
product on $\n$ induced by $ds^2$. If $\ip|_{[\n.\n]}$ is nondegenerate 
then $N$ is abelian or $2$-step nilpotent.
\end{theorem}

If the $GO$ condition in Theorems~\ref{th:dernondeg} and 
\ref{th:lorentzdernondeg} is replaced by the
natural reductivity condition, the complete description of all resulting 
nilmanifolds is given in \cite[Theorem~3.2]{Ova2}. There the construction in 
arbitrary signature is similar to the construction for Riemannian signature.

\section{The Double Extension Theorem}
\label{theorem2}

Given a metric Lie algebra $\m_0$ with a nondegenerate inner product 
$\ip_0$, say of signature $(p,q)$, let $\m_1$ be its central extension,
as in the exact sequence
\begin{equation}\label{m1}
0 \to \br e \to \m_1 \to \m_0 \to 0 \text{ where } 
	\m_1 = \br e \oplus \m_0
\text{ with $e \ne 0$ central in $\m_1$}
\end{equation}
where the arrows are Lie algebra homomorphisms.
Let a Lie algebra $\m_2=\br f \oplus \m_1$ be an extension of $\m_1$ by a 
nonzero derivation $f$ as follows.
\begin{equation}\label{m2}
0 \to \m_1 \to \m_2 \to \br f \to 0, \text{ Lie algebra exact sequence, }
\ad_{\m_2} (f)|_{\m_1} \in \Der{\m_1}.
\end{equation}
Then $(\m_2,\ip)$ is a Lie algebra with nondegenerate inner product $\ip$
of signature $(p+1,q+1)$ defined by 
\begin{equation}\label{m3}
\ip|_{\m_0}=\ip_0,\, \<e,\m_0\>=\<f,\m_0\>=0, \, \|e\|=\|f\|=0, 
\, \<e,f\>=1.
\end{equation}  
In particular, if $\ip_0$ is positive definite, so $\ip$ is of Lorentz
signature, then $(\m_2,\ip)$ is called
called the \emph{Lorentz double extension} of 
$(\m_0, \ip_0)$. The double extension is a well-known tool since in 
\cite{MR} it has been used for constructing bi-invariant pseudo-Riemannian 
inner products. Our approach is closer to that of \cite{YD}.

\begin{theorem}\label{th:derdeg} 
Let $(M = G/H, ds^2)$ be a connected Lorentz geodesic orbit nilmanifold, 
where $G = N \rtimes H$ with $N$ nilpotent. Let $\ip$ be the inner product 
on $\n$ defined by $ds^2$. Suppose that $\ip|_{[\n,\n]}$ is \emph{degenerate}. 
Then $(\n,\ip)$ is a Lorentz double extension of the metric Lie algebra 
corresponding to a Riemannian $GO$ nilmanifold (which necessarily 
is abelian or $2$-step nilpotent).
\end{theorem}

\begin{proof}
Suppose the restriction of $\ip$ to $\n'=[\n,\n]$ is degenerate. Let $e \in \n'$ be a nonzero null vector. Let $\vg$ denote the orthogonal
complement to $\n'$ in $\n$, and let $\m_1 = \n' + \vg$. We have $\n' \cap \vg = \br e$ and $\m_1 = e^\perp$.  Moreover, all four subspaces $\br e, \n', \vg$ and $\m_1$ are $\ad_\g(\h)$-invariant.

The subspace $\m_1$ is a degenerate hyperplane in $\n$ and is an $\ad_\g(\h)$-invariant ideal. We choose a null vector $f$ such that
$\br f \oplus \m_1 = \n$ and $\<f,e\>=1$ (note that the choice of such $f$ is not unique). Clearly $\m_1$ is $\ad_\g(f)$-invariant and $\ad_\g(f)|_{\m_1}$ acts on $\m_1$ as a nilpotent derivation. Moreover, the restriction of $\ip$ to $\Span(f,e)$ is Lorentz. Let $\m_0=(\Span(f,e))^\perp$. Define the inner product $\ip_0$ on $\m_0$ to be the restriction of $\ip$ to $\m_0$\,. Note that $\ip_0$ is positive definite.

According to our definition, to prove that $(\n,\ip)$ is a Lorentz double extension of $(\m_0, \ip_0)$, it remains to show that $e$ lies in the center of $\m_1$ (and then the Lie bracket $[\cdot,\cdot]_0$ on $\m_0$ is defined by requiring that $(\m_0,[\cdot,\cdot]_0)$ is isomorphic to the quotient algebra $\m_1/(\br e)$). To see that $e$ is central in $\m_1$, take $T = X \in \m_0$ and $T'=e$ in \eqref{eq:golemma}. Using the facts that $e \perp \m_0$ and that $\br e$ is $\ad_\g(\h)$-invariant, we obtain $\<\ad_\g(e) X, X\>=0$ for all $X \in \m_0$\,. Then for $X \in \m_0$ we have $\ad_\g(e) X = KX +\mu(X)e$ for some endomorphism $K$ of $\m_0$ and a linear form $\mu$ on $\m_0$. As $K$ is nilpotent and skew-symmetric relative to a positive definite inner product $\ip_0$, we get $K=0$, and so $[X,e]=-\mu(X)e$, for all $X \in \m_0$\,. Since $\ad_\g(X)$ is nilpotent, $e$ lies in the center of $\m_1$.

To complete the proof of the theorem we need to show that the metric Lie algebra $(\m_0, \ip_0)$ is the Lie algebra of a Riemannian GO nilmanifold.

Take a nonzero $X \in \m_0$. From the Geodesic Lemma, we can find $A(X) \in \h$ such that \eqref{eq:golemma} with $T = X$ holds for all $T' \in \n$, and in particular, for all $T'=Y \in \m_0$\,. Note that $k(X) = 0$, as $X$ is non-null. Define a skew-symmetric operator $D(X)$ on $(\m_0, \ip_0)$ by
$\<D(X)Y,Y'\>_0 = \<[A(X),Y],Y'\>$ for $Y,Y' \in \m_0$. As $\m_1$ and $\br e$ are $\ad_\g(A(X))$-invariant and $[e,\m_1]=0$, we find that $D(X)$ is a (skew-symmetric) derivation of the algebra $(\m_0,[\cdot,\cdot]_0)$. Then \eqref{eq:golemma} gives $\<[X,Y]_0,X\>_0 + \<D(X)Y,X\>_0=0$ (as $\<e,\m_1\>=0$) which completes the proof according to the Riemannian version of the Geodesic Lemma.
\end{proof}

\section{Examples Related To Degree of Nilpotence}
\label{examples}

In this section we use Theorem~\ref{th:derdeg} to show that the set of 
nilpotency steps of Lorentz $GO$ nilmanifolds is unbounded.
In our examples the group $G=N \rtimes H$ where $H = H(N)$ is the full 
group of isometric automorphisms of $(N,ds^2)$.  We construct both 
naturally reductive examples and examples that are not naturally reductive, 
first constructing the naturally reductive ones and then modifying them 
to get examples that are not naturally reductive.

Let $d > 1$ and let $S$ be an $s \times s$ matrix that is $d$-step nilpotent,
that is, $S^d=0$, but $S^{d-1} \ne 0$. For example $S$ could be
the $(d+1)\times (d+1)$ matrix $(S_{i,j})$ where $S_{i,i+1} = 1$ for
$1 \leq i \leq d$ and all other $S_{i,j} = 0$.  Introduce the following 
$(2s)\times(2s)$ matrices:
  \begin{equation}\label{eq:PQ}
  P = \left ( \begin{smallmatrix} 0_s & -S^t \\ 
	S & S^t-S\end{smallmatrix}\right ) \text{ and } Q = 
	\left ( \begin{smallmatrix} 0_s & S^t \\ S & S^t-S\end{smallmatrix}
	\right ).
  \end{equation}
The matrix $P$ is skew-symmetric and the matrix $Q$ is nilpotent. To see the 
latter, compute $Q=\left(\begin{smallmatrix} I_s & I_s \\ 0_s & I_s 
\end{smallmatrix}\right)\left(\begin{smallmatrix} -S & 0_s \\ S & S^t 
\end{smallmatrix}\right)\left(\begin{smallmatrix} I_s & -I_s \\ 0_s & I_s 
\end{smallmatrix}\right)$.  From that, $Q$ is nilpotent and $Q^{d-1} \ne 0$.

Following the idea of the construction in Theorem~\ref{th:derdeg} we 
start with a 
Riemannian metric Lie algebra $(\m_0,\ip_0)$.  Here $\m_0 = \br^{2s}$
as a vector space, $\ip_0$ is positive definite, and we fix an 
orthonormal basis $\cB$.  
The Lie algebra structure of $\m_0$ and a central extension 
$\m_1 = \br e \oplus \m_0$ (vector space direct sum), 
$0 \to \br e \to \m_1 \to \m_0 \to 0$, are given by
 \begin{equation} \label{eq:exmtoeg}
 \begin{aligned}
  &\<e,\m_1\>=0 \text{ and } \<X,Y\>=\<X,Y\>_0 \text{ for } X,Y \in \m_0 \\
  &[e,\m_1]=0 \text{ and } [X,Y]=\<PX,Y\>_0 \,e \text{ for } X,Y \in \m_0
 \end{aligned}
 \end{equation}
where $P$ has matrix \eqref{eq:PQ} relative to the basis $\cB$ of $\m_0$. 
Next, define the extension $\n$ of $\m_1$ by $\n=\br f \oplus \m_1$ (vector 
space direct sum), $0 \to \m_1 \to \n \to \br f \to 0$, 
with the Lie bracket and the inner product defined by 
\eqref{eq:exmtoeg} on $\m_1$, and additionally, by
 \begin{equation} \label{eq:exegton}
 \begin{aligned}
    &\<f,e\>=1 \text{ and } \<f,X\> = \<f,f\> = 0 \text{ for } X \in \m_0 \\
    &[f,e]=0 \text{ and } [f,X]=QX \text{ for } X \in \m_0
 \end{aligned}
 \end{equation}
with matrices relative to $\cB$ as before. The algebra $\n$ so constructed is 
nilpotent, and is of step at least $d$, as $\ad_\g(f)^{d-1}X = Q^{d-1}X \ne 0$ 
for some $X \in \m_0$.  Note $\dim \n = 2s+2$.  

We claim that $(\n,\ip)$ is geodesic orbit. To see this, 
let $T=\alpha f + X + \eta e \in \n$ where $X \in \m_0$ and 
$\alpha, \eta \in \br$.  Define $k(T)=0$ and $A(T) \in \h = {\rm Lie}(H)$  
in such a way that
 \begin{equation} \label{eq:AT}
 \begin{aligned}
  \ad_\g(A(T))& e = 0,\; \ad_\g(A(T)) f = v(T), \text{ and } 
	\ad_\g(A(T)) Y = -\<v(T),Y\>e \text{  for }Y \in \m_0,\\
  &\text{where } v(T)=(Q^t+P)X = 
	\left ( \begin{smallmatrix} 0_s & 0_s \\ 2S & 0_s
	\end{smallmatrix} \right ) X \in \m_0.
  \end{aligned}
  \end{equation}
It is easy to check that $\ad_\g(A(T))$ so defined is a skew-symmetric 
derivation of $(\n,\ip)$.

Using \eqref{eq:exmtoeg}, \eqref{eq:exegton} and \eqref{eq:AT}, for an 
arbitrary $T'=\beta f + Y + \rho e \in \n$ with $Y \in \m_0$ and 
$\beta, \rho \in \br$, we have 
$$
[T,T']=\alpha QY - \beta QX + \<PX,Y\>e \text{ and } 
	[A(T), T']=\beta v(T) - \<v(T),Y\>e,
$$ 
and so
 \begin{align*} 
  \<[T+A(T),T'],T\> 
	&= \<(\alpha QY - \beta QX + \beta v(T)) + (\<PX,Y\> - 
		\<v(T),Y\>)e,\alpha f + X + \eta e\> \\
  	&= \beta \<v(T)-QX,X\> +\alpha \<(Q^t+P)X-v(T),Y\>=0.
  \end{align*}
By the Geodesic Lemma, $(\n,\ip)$ is $GO$.
  
In fact $(\n,\ip)$ is ($G$-)naturally reductive. The geodesic graph $A(T)$ 
given by \eqref{eq:AT} is linear in $T$ (and $k=0$), so according to 
Proposition~\ref{p:natred}, we need to check that $A(T)$ is 
$\ad_\g(\h)$-equivariant, where $\h$ is the Lie algebra of skew-symmetric 
derivations of $(\n,\ip)$. Straightforward computation using \eqref{eq:exmtoeg} 
and \eqref{eq:exegton} shows that $\h$ is spanned by elements $B$ such that
  \begin{equation}\label{eq:exderalg}
  \begin{aligned}
    u:=&\ad_\g(B) f \in \m, \ad_\g(B) e = 0 \text{ and }
	\ad_\g(B) X = \Phi X - \<u, X\> e  
	\text{ for } X \in \m_0, \\
    &\text{ where } \Phi \in \so(\m_0) \text{ and } [\Phi ,P]=[\Phi,Q]=0 
	\text{ and } (Q^t+P)u=0.
  \end{aligned}
  \end{equation}
Then, whenever $T \in \n$ and $B \in \h$, we have 
$\ad_\g(A)(\ad_\g(B) T) = \ad_\g(B) \ad_\g(A(T))$ by a direct calculation, using 
\eqref{eq:AT}, \eqref{eq:exderalg} and the consequence $(Q^t+P)^2=0$ of 
\eqref{eq:PQ}.  This completes the proof that $(\n,\ip)$ is ($G$-)naturally 
reductive.

A manifold that is not naturally reductive can be constructed by taking the 
direct sum 
of an algebra constructed above and one that is not naturally reductive.
For example let $(\n_1, \ip_1)$ be a Lorentz 
algebra constructed above, and let $(\n_2, \ip_2)$ be a Riemannian $2$-step 
nilpotent Lie algebra defined by $\n_2 = \z \oplus \ag$, $\ag=\z^\perp=\br^4$, 
with Lie bracket defined as follows.  Let $\{z_1, z_2\}$ is an orthonormal 
basis for $\z$ and $J_1, J_2 \in \so(4)$ such that 
$J_iJ_j+J_jJ_i=-2\delta_{ij} I_4$ for $i,j=1,2$.  So $J_1$ and $J_2$ lie in 
the same $\so(3)$ factor of $\so(4)$ and are orthonormal relative to the 
Killing form, up to scale.  Then the Lie algebra structure on $\n_2$ is
given by
\begin{equation*}
  [\n_2,\z] = 0 \text{ and } [X,Y] = \<J_1X,Y\>_2 z_1 + \<J_2X,Y\>_2 z_2, 
	\text{ for } X, Y \in \ag.
  \end{equation*}
Then $(\n_2, \ip_2)$ is $GO$, but not naturally reductive 
(\cite[Proposition~3]{Kap}, \cite[Theorem~5.3(I)]{KV}). 
Note $\dim \n_2 = 6$.

Now define the Lorentz Lie algebra 
$(\n, \ip) = (\n_1,\ip_1) \oplus (\n_2,\ip_2)$ as the orthogonal direct
sum of $(\n_1,\ip_1)$ and $(\n_2,\ip_2)$.  Let $\pi_i: \n \to \n_i$ denote 
the orthogonal projections.  Let $H$ be the full group 
of skew-symmetric automorphisms of the resulting nilmanifold $(N,ds^2)$ and 
let $\h$ its Lie algebra. Note that $\h$ may a priori be bigger than 
the direct sum of the corresponding algebras for  $(\n_1,\ip_1)$ and 
$(\n_2,\ip_2)$.  But if $B \in \h$ then $(\pi_i B)|_{\n_i}$ is a skew symmetric 
derivation of $(\n_i,\ip_i)$.  Were $(\n,\ip)$ naturally reductive, it would
admit a linear geodesic graph $A: \n \to \h$ by Proposition~\ref{p:natred}, 
which would then give a linear geodesic graph $(\pi_2 A)|_{\n_2}$ for 
$(\n_2,\ip_2)$, in contradiction with Proposition~\ref{p:natred} and the
comment after its proof.

It would be very interesting to see which of these Lorentz $GO$ nilmanifolds 
are weakly symmetric.

We summarize the considerations of this section as they apply to the
cases where the nilpotent matrices $S$ are exemplified just before 
(\ref{eq:PQ}).

\begin{theorem}\label{th:examples}
For every number $d > 1$
there exist both a naturally reductive $GO$ Lorentz nilmanifold 
$(N,ds^2)$ nilpotent of step $\geq d$ and dimension $2d+4$, and a 
$GO$ Lorentz $d'$-step nilmanifold 
$(N_1,ds_1^2) \times (N_2,ds_2^2)$ with $d' \ge d$
and dimension $2d+10$.
\end{theorem}

Note that the $GO$ Lorentz nilmanifolds constructed in Section \ref{examples},
in particular those cited in Theorem \ref{th:examples}, have nilpotence
bounds in sharp contrast to the bounds of Theorem \ref{th:dernondeg} and
the results of \cite{CWZ}.
We now compare the constructions of this section, including the case
of Theorem \ref{th:examples}, with the nilpotence bounds of  
\cite[Theorem 8]{CWZ}; there, if $\ip|_{[\n,\n]}$ is degenerate and 
the action of $\ad_\g(\h)$ on $\n$ is completely reducible
(semisimple), then $N$ is abelian or $2$-step nilpotent.  
The reason for the difference here is

\begin{proposition}\label{examples-unip}
The $GO$ Lorentz nilmanifolds constructed in Section \ref{examples} have the 
property that the action of ${\rm Ad}_G(H)$ on $\n$ is not completely reducible.
\end{proposition}

\begin{proof}
Fix a Lorentz nilmanifold as constructed in this section.
If we assume complete reducibility we obtain 
a contradiction as follows.  First note from (\ref{eq:PQ}),
(\ref{eq:exmtoeg}) and (\ref{eq:exegton}), that 
$[\n,\n] = Q\m_0 + \br e$ is ${\rm Ad}_G(H)$-invariant.  Thus
$[\n,\n] \cap [\n,\n]^\perp = \br e$ is ${\rm Ad}_G(H)$-invariant.
Now $(\br e)^\perp = \m_0 + \br e = \m_1$ is stable under ${\rm Ad}_G(H)$.
An invariant complement to $\m_1$ in $\n$ may be taken to be $\br f$.
So now $\m_0$, $\br e$ and $\br f$ are ${\rm Ad}_G(H)$-invariant, while
(\ref{eq:AT}) shows that some elements of $\ad_\g(\h)$ map $f$
into $\m_0$, $\m_0$ into $\br e$, and $e$ to $0$, with nonzero images.  
That contradicts complete reducibility of the action of ${\rm Ad}_G(H)$ on $\n$.
\end{proof}

\section{Remarks}\label{remarks}

Together, Theorem \ref{th:examples} and Proposition \ref{examples-unip} show
that the existence of a reductive decomposition 
$\g = \h \oplus \n$ is crucial in Theorems
\ref{th:dernondeg} and \ref{th:derdeg}.  This gives a good indication of the
difficulty of finding structural results for non-reductive $GO$ Lorentz 
nilmanifolds.  However it might be worthwhile to explore two special cases:
naturally reductive and weakly symmetric.

Recall that a pseudo-Riemannian manifold $(M,ds^2)$ is {\em weakly symmetric}
if, given $x \in M$ and a tangent vector $\xi \in M_x$\,, there is an isometry
$s_{x,\xi}$ of $(M,ds^2)$ such that $s_{x,\xi}(x) = x$ and $ds_{x,\xi}(\xi)
= \xi$; $(M,ds^2)$ is {\em symmetric} if we can always choose $s_{x,\xi}$
independent of $\xi$.  Let $(M,ds^2)$ be weakly symmetric and $G = I(M,ds^2)$
with $M = G/H$. 
In the Riemannian case \cite[Theorem 13.1.1]{W2007} the nilradical $N$ 
of $G$ is abelian or $2$-step nilpotent.  In general, if there is a
reductive decomposition $\g = \m + \h$ with $\n \subset \m$ and the metric
definite on $[\n,\n]$ then \cite[Theorem 4.12]{CW2012} $N$ is abelian or 
$2$-step nilpotent.  There $N$ does not have to be transitive on $M$.  This
suggests that Theorem \ref{th:dernondeg} might apply when $N$
is Lorentz, or perhaps even in general signature of
metric when the inner product on $[\n,\n]$ is definite or Lorentz.

One can consider the possibility of a converse to Theorem \ref{th:derdeg},
perhaps guided by the examples of Section \ref{examples}.
Let $(\n,\ip)$ be a double extension of a metric Lie algebra $(\n_0,\ip_0)$
corresponding to a Riemannian $GO$ manifold.  What are the conditions for
$(\n,\ip)$ to be $GO$?  Or naturally reductive?  Or weakly symmetric?  And
what if $(\n_0,\ip_0)$ corresponds to a Lorentz $GO$ manifold?


\begin{thebibliography}{CWZ}

\bibitem[BN]{BN}
V.~N.~Berestovskii \& Yu.~G.~Nikonorov, \emph{Riemannian Manifolds and 
Homogeneous Geodesics}. Springer Monographs in Mathematics, Springer, 
Cham, 2020.

\bibitem[CZ]{CZ}
G.~Calvaruso \& A.~Zaeim, \emph{Four-dimensional pseudo-Riemannian g.o. 
spaces and manifolds}, J. Geom. Phys. \textbf{130} (2018), 63-80.

\bibitem[CFZ]{CFZ}
G.~Calvaruso, A.~Fino \& A.~Zaeim, \emph{Homogeneous geodesics of non-reductive 
homogeneous pseudo-Riemannian 4-manifolds}, Bull. Braz. Math. Soc. 
\textbf{46} (2015), 23-64.

\bibitem[CW]{CW} 
Z. Chen \& J. A. Wolf, \emph{
Weakly Symmetric Pseudo--Riemannian Nilmanifolds}, to appear in J.
Diff. Geom. \url{https://arxiv.org/abs/1806.07847.pdf}

\bibitem[CNN]{CNN}
Z.~Chen, Y.~Nikolayevsky \& Yu.~Nikonorov, \emph{Compact geodesic orbit 
spaces with a simple isotropy group}.
\url{https://arxiv.org/pdf/2009.06329.pdf}

\bibitem[CWZ]{CWZ}
Z.~Chen, J.~A.~Wolf \& S.~Zhang, \emph{On the geodesic orbit property for 
Lorentz manifolds}, J. Geom. Anal. \textbf{32} (2022).

\bibitem[CW2012]{CW2012} Z. Chen \& J. A. Wolf, 
\emph{Pseudo-Riemannian weakly symmetric manifolds}, Ann. Global Anal. 
Geom. {\bf 41} (2012), 381-390.

\bibitem[dBO]{dBO}
V.~del~Barco \& G.~Ovando, \emph{Isometric actions on pseudo-Riemannian 
nilmanifolds}, Ann. Glob. Anal. Geom. \textbf{45} (2014), 95-110.

\bibitem[DZ]{DZ}
J.~E.~D'Atri \& W.~Ziller, \emph{Naturally reductive metrics and Einstein 
metrics on compact Lie groups}, Mem. Amer. Math. Soc. 18, No. 215 (1979).

\bibitem[DK]{DK}
Z.~Du\v{s}ek \& O.~Kowalski, \emph{Light-like homogeneous geodesics and 
the geodesic lemma for any signature}, Publ. Math. Debrecen, \textbf{71} 
(2007), 245-252.

\bibitem[Gor]{Gor}
C.~S.~Gordon, \emph{Homogeneous Riemannian manifolds whose geodesics are 
orbits}, in ``Topics in Geometry, in memory of Joseph D'Atri'', Progress 
in Nonlinear Differential Equations 20, Birkh\"{a}user, Basel, (1996), 155-174.

\bibitem[Kap]{Kap}
A.~Kaplan, \emph{On the geometry of groups of Heisenberg type}, Bull. London 
Math. Soc. \textbf{15} (1983), 35-42.

\bibitem[Kos]{Kos}
B.~Kostant, \emph{On differential geometry and homogeneous spaces II}, Proc. 
N.A.S. \textbf{42} (1956) 354-357.

\bibitem[KV]{KV}
O.~Kowalski \& L.~Vanhecke, \emph{Riemannian manifolds with homogeneous 
geodesics}, Boll. Un. Math. Ital. B  \textbf{5}(7) (1991), 189-246.

\bibitem[MR]{MR}
A.~Medina \& P.~Revoy, Alg\`{e}ebres de Lie et produit scalaire invariant, 
Ann. scient. \'{E}c. Norm. Sup. \textbf{18} 1985, 553-561.

\bibitem[Mos]{Mos}
G.~D.~Mostow, \emph{On maximal subgroups of real Lie groups}, Ann. of Math. (2) \textbf{74} (1961), 503-517.

\bibitem[Ova1]{Ova1}
G.~Ovando, \emph{Naturally reductive pseudo-Riemannian spaces}, J. Geom. 
Phys. \textbf{61} (2011), 157-171.

\bibitem[Ova2]{Ova2}
G. Ovando, \emph{Naturally reductive pseudo-Riemannian $2$-step nilpotent
Lie groups}, Houston J. Math \textbf{39} (2013), 147-167.

\bibitem[Sze]{Sze}
J.~Szenthe, \emph{Sur la connection naturelle \`{a} torsion nulle}, Acta 
Sci. Math. (Szeged) \textbf{38} (1976), 383-398.

\bibitem[W1963]{W1963}
J.~A.~Wolf, \emph{On Locally Symmetric Spaces of Non-negative Curvature and 
Certain Other Locally Homogeneous Spaces}, Comm. Math. Helv. \textbf{37} 
(1963), 265-295.

\bibitem[W2007]{W2007} J. A. Wolf, ``Harmonic Analysis on Commutative 
Spaces'', Mathematical Surveys and Monographs, 142.  American Mathematical 
Society, 2007.

\bibitem[YD]{YD}
Z.~Yan \& S.~Deng, \emph{Double Extensions on Riemannian Ricci Nilsolitons}, 
J. Geom. Anal. \textbf{31} (2021), 9996-10023.
\end{thebibliography}
\end{document}